\documentclass[12 pt]{article}
\usepackage[english]{babel}
\usepackage[utf8]{inputenc}
\usepackage{johd}
\usepackage{natbib}
\usepackage{multirow}

\usepackage{geometry}
 \usepackage[ruled,vlined]{algorithm2e}
 
\usepackage[shortlabels]{enumitem}
\usepackage{graphicx,  amssymb,amsmath,amsthm, amsfonts, tikz} 
\usepackage[font=small,labelfont=bf]{caption}
\usepackage{subcaption}
\everymath{\displaystyle}

\newtheorem{theorem}{Theorem}

\newtheorem{remark}{Remark}

\newcommand{\K}{\mathcal{K}}

\newcommand{\per}{\mathrm{per}}

\newcommand{\F}{\mathcal{F}}
\renewcommand{\S}{\mathcal{S}}
\newcommand{\I}{{\boldsymbol{I}}}

\renewcommand{\Im}{\mathrm{Im}\,}

\newcommand*{\N}{\ensuremath{\mathbb{N}}}
\newcommand*{\Z}{\ensuremath{\mathbb{Z}}}
\newcommand*{\R}{\ensuremath{\mathbb{R}}}

\newcommand{\x}{{\boldsymbol{x}}}

\newcommand{\J}{{\boldsymbol{J}}}
\newcommand{\txt}{\textstyle}

\title{Identifying defective units in infinite periodic arrays of point sources}

 \author{Dinh-Liem Nguyen\thanks{Department of Mathematics, Kansas State University, Manhattan, KS 66506 (\href{mailto:dlnguyen@ksu.edu}{dlnguyen@ksu.edu})} \and  Nhung H. Nguyen\thanks{Department of Mathematics, Kansas State University, Manhattan, KS 66506 (\href{mailto:nhungnh@ksu.edu}{nhungnh@ksu.edu})} \and  Thi-Phong Nguyen\thanks{Department of Mathematical Sciences, New Jersey Institute of Technology, Newark, NJ 07102 (\href{mailto:thiphong.nguyen@njit.edu}{thiphong.nguyen@njit.edu})}}

 \date{}

\begin{document}
\maketitle
\begin{abstract} 
\noindent This paper focuses on identifying defective units in unbounded periodic arrays of point sources using boundary data. The study is motivated by the noninvasive evaluation of large-scale periodic source systems.
Unlike classical inverse source problems in free space, the key challenge here lies in the disruption of periodicity caused by defective sources in the infinite array. To address this, we employ the Floquet–Bloch transform to reformulate the original inverse source problem as a quasi-periodic inverse source problem.
We first establish uniqueness theorems for both the original and the quasi-periodic formulations. Then, we develop a new numerical method for identifying defective sources. This method combines a sampling indicator function with an algebraic technique to determine not only the number of defective sources, but also their locations and intensities. Numerical experiments are presented to validate the effectiveness of the proposed method.

\end{abstract}

\noindent \\

\section{Introduction}

We consider a one-dimensional array of point sources that is unbounded, periodic along the $x_1$-direction with period $L>0$, and bounded in the $x_2$-direction between $-h$ and $h$, for some $h>0$. Each unit cell of length 
$L$ contains one point source. We denote by 
\begin{equation}\notag
\Omega: = \R \times (-h,h), \quad \Omega_0: = \left(-\frac{L}{2}, \frac{L}{2}\right)\times (-h,h),
\end{equation} 
and let $\x_0$ be the location of the source in $\Omega_0$. The locations of all other sources in the array are then given by $\x_j :=\x_0+(jL,0)^{\top}$ for $j\in \Z$. These sources are represented by the delta distribution $\delta(x;\x_0+(jL,0)^{\top})$ for $x\in \Omega$ and have corresponding intensities $\gamma_j$.
\begin{center}
    \includegraphics[width=0.95\textwidth]{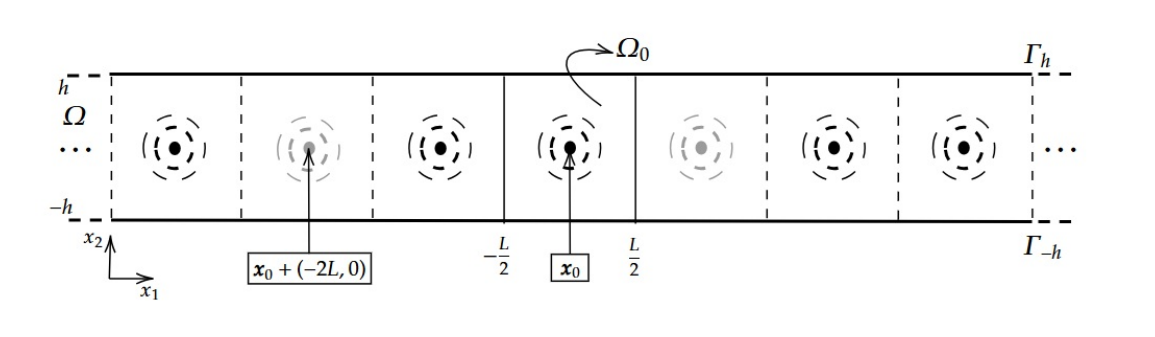}
\end{center}
In this work, we assume that some sources are not functioning properly, in the sense that their intensities differ from $1$. In particular, there are $N \in \N^{*}$ defective sources located at $\x_{m_\ell}$, each having intensity  $\gamma_{m_\ell}\neq 1$,  for $\ell = 1, \ldots, N$. We denote by $\I$ the set of indices of all defective sources, i.e., $\I:=\{m: \gamma_m\neq 1\}$ and assume that $\I$ is a finite set. 
\smallskip 

The wave generated by all point sources in a homogeneous background with a wave number $k > 0$ satisfies the Helmholtz equation
\begin{equation} \notag
\Delta u + k^2 u = - \left(\sum_{m\in \Z \setminus \I} \delta(\cdot, \x_0+(mL,0)^{\top}) + \sum_{m \in \I} \gamma_m \delta(\cdot, \x_0+(mL,0)^{\top})\right),
\end{equation} 
which is equivalent to
\begin{equation}
\label{eq:direct:1s0}
\Delta u + k^2 u = -\left(\delta_{\per}(\cdot, \x_0) +  \sum_{m \in \I} \sigma_m \delta(\cdot, \x_m)\right),
\end{equation}
where 
$$\sigma_m := \gamma_m -1, \quad \delta_{\per}(\cdot, \x_{0}): = \sum\limits_{m \in \Z} \delta(\cdot, \x_0+(mL,0)^{\top}),\qquad \text{in } \Omega.$$
We denote by $G_{\per}(\cdot, \x_0)$ the outgoing  periodic Green's function with period $L$ associated with the point source at $\x_0$; that is, $G_{\per}(\cdot, \x_0)$ admits the spectral representation 
\[
    G_{\per}(x,\x_0) =\sum_{n\in\Z}\frac{i}{2L\theta_{n}}e^{i\frac{2\pi n}{L}\left(x_1-\x_{0}^{(1)}\right)+i\theta_{n}\left|x_2- \x_{0}^{(2)}\right|},\quad x, \x_0 \in \Omega_0, \;  x\neq \x_0,
\] 
where $\theta_{n} = \sqrt{k^2 - \left(2\pi n/L\right)^2}, \, \theta_{n}\neq 0$, with $\Im(\theta_n) \geq 0$.
Then we impose the outgoing behavior of $u$ by requiring that $ u - G_{\per}(\cdot, \x_0)$ satisfies the Sommerfeld radiation condition.  Then the outgoing wave solution $v$ to~\eqref{eq:direct:1s0} is  given by 
\begin{equation}
\label{form:v}
    u(x) = G_{\per}(x, \x_0) + \sum_{m \in \I} \sigma_m\frac{i}{4} H^{(1)}_0(k|x - \x_m|),
\end{equation}where $H^{(1)}_0$ is a Hankel function of the first kind and zero order.

\smallskip

The objective of this work is to determine the set of indices of all defective sources $\I$ and the corresponding intensities $\gamma_m$ for all $m \in \I$ using  measurements of $v$ on  horizontal lines $\Gamma_{\pm t}: =\R\times \{\pm t\}$ for some $t > h$. 
We consider the following inverse problem below.

\smallskip

\noindent {\bf Inverse problem:} Determine the indices $m \in  \I$ corresponding to the defective sources located at $\x_m = \x_0 + (mL,0)^{\top}$  and their intensities $\gamma_m$ using measurements of the radiated waves $v$ on $\Gamma_{ t} \cup \Gamma_{-t}$ with $t > h$.

\smallskip

We note that the results in this paper can be directly extended to the 3D setting with two-dimensional periodic source arrays.
This problem arises when the source distribution deviates from its expected or designed configuration. Understanding and locating such defects is crucial for diagnosing system failures, improving performance, or reconstructing hidden inhomogeneities. Our goal is to develop a robust and computationally efficient method capable of detecting and localizing defective units in large scale periodic source systems using single frequency measurements. Unlike classical inverse source problems in free space or bounded domains, the key challenge here lies in the disruption of periodicity caused by defective sources in the infinite array. To address this, we employ the Floquet–Bloch transform to reformulate the inverse problem of interest as a quasi-periodic inverse source problem defined on a single unit cell, thereby avoiding the need to solve the inverse source problem on an unbounded domain~\cite{haddar2017volume, lechleiter2017convergent}.

Over the past two decades, inverse source problems in free space or bounded domains have been studied extensively from both theoretical and computational perspectives. Significant progress has been made in the analysis of uniqueness and stability properties, see, e.g., \cite{acosta2012multi, alves2009iterative, alzaalig2020fast, bao2020stability, bao2010multi,  cheng2016increasing, el2005inverse,el2013stability, el2000inverse} and the references therein. Meanwhile, various numerical algorithms have been proposed to reconstruct sources from boundary data. In particular, methods based on multi-frequency data have shown enhanced stability in recovering unknown sources, see \cite{bao2011numerical, bao2015recursive,  chung2009identification, el2011inverse,  eller2009acoustic, griesmaier2017factorization,ji2021reconstruction,  nguyen2019convergent,nguyen2022reconstructing,ren2019imaging, zhang2015fourier}. 

Inverse source problems with single-frequency data can face non-uniqueness of solution unless certain a priori assumptions about the source are imposed~\cite{kress2013reconstruction}. To address this challenge, several computational strategies have been developed. Algebraic reconstruction methods have been applied to identify point sources in elliptic equations such as the Poisson and Helmholtz equations \cite{acosta2012multi, li2021lipschitz, nara2008algebraic}, while Newton-type iterative algorithms have been used to recover both point and extended sources~\cite{ji2021reconstruction}. 
In recent years, sampling-based techniques have gained popularity due to their simplicity and efficiency.  Sampling-type methods have been studied in~\cite{harris2023reconstruction,harris2024direct, zhang2018locating} for identifying point sources in the Helmholtz equation and Maxwell's equations. The key advantages of such methods are their non-iterative nature, robustness, speed, and low computational cost.

Unlike the classical inverse source problems considered in the works mentioned above, the problem of identifying defective sources within an infinite periodic array appears, to the best of our knowledge, not to have been studied before. In this paper, we first establish uniqueness results for both the original inverse problem and its quasi-periodic counterpart. Motivated by the numerical method studied in~\cite{harris2023reconstruction} for the inverse point-source problem and the approach developed in~\cite{nguyen2023new} for imaging periodic scattering structures, we introduce a new numerical reconstruction scheme that combines a novel imaging function with an algebraic procedure for detecting defective units in unbounded periodic source arrays. The algebraic procedure employs linear algebra tools. The proposed method is stable with respect to noise, computationally efficient, and straightforward to implement.

{
It is worth emphasizing that the algebraic approach employed here identifies defective sources through the evaluation of matrix determinants and the solvability of a system of equations derived from their relationship with the imaging function. An alternative algebraic, matrix-based approach for recovering source information was investigated in \cite{el2013stability,el2000inverse} for inverse point source problems in bounded domains. In those works, the authors employed the reciprocity gap functional together with specially designed test functions to construct Hankel matrices and associated linear systems to identify unknown sources. Although our approach also relies on matrices and linear systems, the key ingredients are derived from a new imaging function inspired by recent work of the first author and collaborators in~\cite{harris2023reconstruction,harris2024direct}.
}

The structure of the paper is as follows. In Section \ref{ssection2}, we focus on the uniqueness of the solution and the reformulation of the problem through the Floquet-Bloch transform. Section \ref{section3} is dedicated to the development of the new numerical reconstruction method for defective sources. Section \ref{section4} presents a numerical study that demonstrates the performance and effectiveness of the proposed method. Finally, in Section \ref{section5}, we conclude and discuss potential
 future work.

\section{Uniqueness of solution and Floquet-Bloch transform based formulation }\label{ssection2}
We begin by establishing the uniqueness of the solution to the inverse problem. We then derive a Floquet-Bloch transform-based formulation of the problem and prove the corresponding uniqueness result.
\begin{theorem}
\label{thm:uniquess_Ips}
Suppose there are two sets of defective sources: $\widehat{\x}_{m}$ with $ m \in \widehat{\J} \subset \Z$, and $\widetilde{\x}_{m}$ with $ m \in \widetilde{\J} \subset \Z$, such that $\widehat{\x}_m \equiv \widetilde{\x}_m$ if $m\in \widehat{\J} \cap \widetilde{\J}$ and the corresponding outgoing solutions $\widehat{u}$ and $\widetilde{u}$ to the following problems  
 \begin{equation}\notag
     \Delta \widehat{u} + k^2 \widehat{u} = - \delta_{\per}(\cdot, \x_0)  -  \sum_{m \in \widehat{\J}} \widehat{\sigma}_{m} \delta(\cdot, \widehat{\x}_m)
 \end{equation}
and \begin{equation}\notag
     \Delta \widetilde{u} + k^2 \widetilde{u} = -\delta_{\per}(\cdot, \x_0)  -   \sum_{m \in \widetilde{\J}} \widetilde{\sigma}_{m} \delta(\cdot, \widetilde{\x}_m)
 \end{equation} coincide on $\Gamma_{t} \cup \Gamma_{-t}$. Then $ \widehat{\J} = \widetilde{\J}$ and $\widehat{\sigma}_m = \widetilde{\sigma}_m$, for all $m \in \widehat{\J}$. 
\end{theorem}

\begin{proof}
{ By eliminating $\delta_{\per}$ from the two equations, the proof can be done using arguments similar to those in the free space case (see \cite{el2013stability, alves2009iterative}). Therefore, we omit the details here.}

\end{proof}

\begin{remark}
      Assume there exists $M>0$ such that all defective sources are contained within $\txt \Omega_M:=\left[ -\frac{ML}{2},\frac{ML}{2}\right]\times [-h,h].$
     The uniqueness of the solution remains valid  if we have Cauchy measurement on 
     $$\Gamma_{\pm t}^M := \left( -\dfrac{ML}{2},\dfrac{ML}{2}\right)\times \{ \pm t\}.$$
     The proof proceeds similarly, employing Holmgren's theorem in the domain $\Omega_M$.
     
\end{remark}

 In the following, we discuss the determination of the location and intensity of the defective sources. This is divided into two steps: the first step focuses on determining the source location within the period $\Omega_0$, and the second step focuses on identifying the set $\I$ of defective indices and their corresponding intensities. The investigation of both steps will be based on the use of Floquet-Bloch transform. 
The one-dimensional Floquet-Bloch transform of a function  
$$\varphi \in \S(\R):=\left\{ \psi \in \mathcal{C}^{\infty}(\mathbb{R}): \, \sup_{x\in \R} |x^{j}  \psi^{(\ell)}(x)| < \infty, \; \; \text{ for all} \; \; j, \ell \in \N \right\}$$ with period $L > 0$,  is defined by
\begin{equation}
\label{def:FBtransform}
\F \varphi(x; \xi): = \sum_{m \in \Z} \varphi(x + mL)e^{- i \xi mL}, \quad \; \text{for all } \xi \in \left(-\frac{\pi}{L}, {\frac{\pi}{L}}\right), \quad x \in \R.
\end{equation}
Note that, for a fixed $\txt \xi \in \left(-\frac{\pi}{L}, {\frac{\pi}{L}}\right)$, the function $\phi(x):= \F \varphi(x; \xi)$ is $\xi-$quasi-periodic with period $L$, i.e.,
$$
\phi(x + mL) = e^{i\xi mL}\phi(x), \; \; \text{for all } m \in \Z, \quad x \in \R.
$$ 
We refer to~\cite{kuchment_floquet_1993, coatleven:pastel-00649212} for a more detailed discussion of the Floquet-Bloch transform and related properties.  Now taking the Floquet-Bloch transform of \eqref{eq:direct:1s0} in the $x_1-$direction and denoting by $u_{\xi} : = \F u (\cdot; \xi),$ $\txt \xi \in \left(-\frac{\pi}{L}, {\frac{\pi}{L}}\right){\setminus  \{0\}},$ we obtain from direct calculations that $u_{\xi}$  satisfies 
\begin{equation}
\label{eq:FB:direct:1s-quasi}
\Delta u_{\xi} + k^2 u_\xi =  - \sum_{m \in \I}  \sigma_m e^{-i\xi mL} \delta(\cdot, \x_0), \quad \text{in } \Omega_0.
\end{equation}
Note that $\xi = 0$ is excluded since $\F\delta_\per(\cdot;\xi) = 0$ for all $\txt \xi \in \left(-\frac{\pi}{L}, {\frac{\pi}{L}}\right){\setminus  \{0\}}$.
In addition, the outgoing behavior of $u$ is now described by the Rayleigh expansion  condition for $u_\xi$ as
\begin{equation}
 \label{Cond:Rayleigh:per}
u_\xi(x) = \sum_{n \in \Z} \hat{u}_{\xi,(n)} e^{i \xi_{(n)} x_1 + i\beta_n|x_2|},\qquad \text{for }|x_2|\ge h,
\end{equation}
where  
$$\xi_{(n)} = \xi+\frac{2\pi}{L} n,\qquad \text{and }\qquad\beta_n = \begin{cases}
    \sqrt{k^2 - \xi_{(n)}^2},\hspace{0.9cm} k^2>\xi_{(n)}^2,\\ \notag
    i\sqrt{ \xi_{(n)}^2-k^2},\qquad k^2<\xi_{(n)}^2,
\end{cases} \qquad n\in \Z,$$
which describes the outgoing behavior of the quasi-periodic wave $u_\xi$. From \eqref{eq:FB:direct:1s-quasi}-\eqref{Cond:Rayleigh:per}, we obtain the following representation of $u_\xi,$
\begin{equation}
\label{sol:FB:direct:1s:quasi}
    u_{\xi}  = \sum_{m \in \I}  \sigma_m e^{-i\xi mL} G_{\xi}(\cdot,\x_0),
\end{equation}
where $G_{\xi}$ is the $\xi-$quasi-periodic Green's function defined as
$$G_\xi(x,y) =\dfrac{i}{2L}\sum_{n\in\Z}\dfrac{1}{\beta_n}e^{i\xi_{(n)}(x_1-y_1)+i\beta_n|x_2-y_2|},\qquad x,y\in \Omega_0, \,  x\neq y.$$
We consider the following quasi-periodic inverse problem.

 \vspace{0.3cm}
  \noindent{\bf Quasi-periodic inverse problem:}  
  Given $u_\xi$ on $$\Lambda_{\pm t}:=\left(-\frac{L}{2}, \frac{L}{2}\right)\times\{\pm t\},$$ for a finite number of $\txt \xi \in \left( -\frac{\pi}{L},\frac{\pi}{L}\right){\setminus  \{0\}}$,  determine the number of defective sources $N$, the indices $m \in \I$ corresponding to their locations $\x_0 + (mL,0)^{\top}$,  along with their intensities $\gamma_m$.

\vspace{0.5 cm}
{ The following theorem establishes the uniqueness of the solution to the problem using a finite number of $\xi$. Unlike the original inverse problem, the dataset used for the quasi-periodic problem under consideration differs from the original measurements because the periodic component has been eliminated. Furthermore, using only a finite number of $\xi$ makes the data even weaker than the dataset.}

\begin{theorem}
\label{thm:uniquenessFB} Suppose there are two sets of indices corresponding to the defective sources $\I_1\subset \Z$ and $\I_2\subset \Z$ with $N=|\I_1 \cup \I_2|$. Consider $\xi_j := \xi_0+j\Delta\xi,$
    with some $0<\Delta\xi\neq \tfrac{2\pi q}{L},\, q\in \mathbb{Q}$ such that $\txt \xi_j\in \left( -\frac{\pi}{L},\tfrac{\pi}{L}\right) {\setminus  \{0\}}$ for all $j=0,1,\ldots,N-1$.  Let $\widetilde{u}_{\xi_j}$ and $\widehat{u}_{\xi_j}$ be the $\xi_j-$quasi-periodic solutions to problem~\eqref{eq:FB:direct:1s-quasi}-\eqref{Cond:Rayleigh:per}, which implies 
\begin{equation}\notag
     \Delta \widetilde{u}_{\xi_j} + k^2 \widetilde{u}_{\xi_j}= - \sum_{n \in \I_1}  \widetilde{\sigma}_n e^{-i\xi_j nL} \delta(\cdot, \x_1), \qquad \text{in } \; \Omega_0
 \end{equation} and 
\begin{equation}\notag
     \Delta \widehat{u}_{\xi_j} + k^2 \widehat{u}_{\xi_j} = - \sum_{n\in \I_2}  \widehat{\sigma}_n e^{-i\xi_j nL} \delta(\cdot, \x_2), \qquad \text{in } \; \Omega_0
 \end{equation} 
 where $\x_1,\x_2\in\Omega_0$. 
    If $\widetilde{u}_{\xi_j}=\widehat{u}_{\xi_j}$ on $\Lambda_{ t} \cup \Lambda_{-t}$ for all $j=0,1,\ldots,N-1$, then $\x_1 =  \x_2$, $\I_1=\I_2$ and $\widetilde{\sigma}_n = \widehat{\sigma}_n$ for all $n\in \I_1$. 
\end{theorem}
\begin{proof}
    The proof of $\x_1 = \x_2$ and 
    \begin{equation}\label{eq:sum}
        \sum_{n \in \I_1}  \widetilde{\sigma}_n e^{-i\xi_j nL} = \sum_{n \in \I_2}  \widehat{\sigma}_n e^{-i\xi_j nL}, \quad \text{for all } j=0,1,\ldots,N-1,
    \end{equation}
     is similar to the proof of Theorem  \ref{thm:uniquess_Ips}, relying on the unique continuation principle. 
    Let $\I = \I_1\cup \I_2$  and denote
    $$\sigma_n = \begin{cases}
        \widetilde{\sigma}_n-\widehat{\sigma}_n,\qquad n\in \I_1 \cap \I_2,\\
        \widetilde{\sigma}_n,\hspace{1.65 cm} n\in \I_1 \setminus \I_2,\\ \notag
        -\widehat{\sigma}_n,\hspace{1.35 cm} n\in \I_2\setminus \I_1.
    \end{cases}$$
    Then \eqref{eq:sum} can be rewritten as
    $$\sum_{n\in \I}\sigma_n e^{-i\xi_j nL}=0,\qquad \text{for all } j=0,1,\ldots,N-1.$$
    Consider the following $N\times N$ matrix
    \begin{align}\label{matrixM}
      & \mathcal{M}=\begin{bmatrix}
e^{-i\xi_0 n_1 L} & e^{-i\xi_0 n_2L}&\ldots & e^{-i\xi_0 n_N L}\\
e^{-i\xi_1 n_1 L} & e^{-i\xi_1 n_2L}&\ldots & e^{-i\xi_1 n_N L}\\
 \vdots & \vdots&\ddots & \vdots \\
 e^{-i\xi_{N-1} n_1 L} & e^{-i\xi_{N-1} n_2L}&\ldots & e^{-i\xi_{N-1} n_N L}
\end{bmatrix}\\ \notag
& \hspace{4.5 cm}
    =\begin{bmatrix}
1 & 1&\ldots & 1\\
r_1 & r_2&\ldots & r_{N}\\
 \vdots & \vdots&\ddots & \vdots \\
 r_1^{N-1} & r_2^{N-1}&\ldots & r_N^{N-1}
\end{bmatrix}\begin{bmatrix}
e^{-i\xi_0 n_1L} & 0&\ldots & 0\\
0 & e^{-i\xi_0 n_2L}&\ldots & 0\\
 \vdots & \vdots&\ddots & \vdots \\
 0 & 0&\ldots & e^{-i\xi_0 n_N L}
\end{bmatrix},
\end{align}
where $r_\ell=e^{-i\Delta\xi n_\ell L}$ with $n_\ell\in \I$ for $\ell=1,\ldots,N$.
    The first matrix on the right-hand side is the transpose of Vandermonde matrix. Moreover, with $\Delta\xi\neq \frac{2\pi q}{L},\,q\in \mathbb{Q}$, $r_\ell$ are distinct since
    $$n_{i_1} \neq n_{i_2}+\dfrac{2\pi}{\Delta\xi L}\rho,\qquad \text{for all }\rho\in \mathbb{Z}, \text{ and }i_1,i_2=1,\ldots,N.$$
    Thus, $\mathcal{M}$ is invertible and so its columns are linearly independent. This implies that $\sigma_n=0$ for all $n\in \I$.
     However, $\widetilde{\sigma}_n$ and $\widehat{\sigma}_n$ are nonzero for all $n\in \I_1$ and all $n\in\I_2$, respectively, so $\I_1 \setminus \I_2$ and $\I_2\setminus \I_1$ must be empty sets.  Therefore, $\I = \I_1=\I_2$ and $\widetilde{\sigma}_n=\widehat{\sigma}_n$ for all $n\in \I$.
\end{proof}

\section{Numerical reconstruction method}\label{section3}
In this section, we develop a new numerical method that combines a stable imaging function and an algebraic technique to recover the number of defects, their locations, and the corresponding intensities in the quasi-periodic inverse problem discussed in the previous section.

\subsection{A stable imaging function for determining $\x_0$}
Recall that $\Omega_0=\left(-\tfrac{L}{2}, \tfrac{L}{2}\right)\times (-h,h)$. Let $\txt \xi \in \left(-\frac{\pi}{L}, \frac{\pi}{L}\right) { \setminus  \{0\}}$. Inspired by~\cite{nguyen2023new}, for $z\in \Omega_0$,  we define the following imaging function
\begin{equation}\label{eq:Iz}
    I_\xi(z)=\int_{\Lambda_t\cup\Lambda_{-t}}  \Phi_\xi(x,z)u_\xi(x)ds(x),
\end{equation}
where 
$$\Phi_\xi(x,z)=\dfrac{-i}{2L}\sum_{j\in \Z \atop \beta_j>0} e^{-i\xi_{(j)}(x_1-z_1)-i\beta_j|x_2-z_2|},\qquad x,z\in \Omega_0.$$
This imaging function aims to determine $\x_0$ with sampling points $z$. The behavior of $I_\xi(z)$ is analyzed in the following theorem.
\begin{theorem}\label{thm:imaging}
    The imaging function satisfies
    \begin{equation}\label{thm:Iz}
        I_\xi(z)=\sum_{n\in \I} \sigma_n e^{-i\xi nL}\K_\xi(z,\x_0),
    \end{equation}
    where 
    $$\K_\xi(z, \x_0) = \sum_{j\in\Z \atop \beta_j >0} \dfrac{1}{2L\beta_j}e^{i\xi_{(j)}\left(z_1-\x_{0}^{(1)}\right)}\cos\left(\beta_j\left(z_2 -\x_{0}^{(2)}\right)\right).$$
   { Furthermore, for $\xi\in \left( -\tfrac{\pi}{L},\tfrac{\pi}{L}\right) \setminus \{0\}$ if 
    $$ k\ge \left|\xi +\frac{2\pi}{L}\right|,\qquad \text{and }\qquad  h<\dfrac{\pi}{2\sqrt{k^2-\xi^2}},$$
     then $|I_\xi(z)|$ attains its global maximum on $\Omega_0$ only at $z=\x_0$.}
\end{theorem}
\begin{proof}
By \eqref{sol:FB:direct:1s:quasi}, we get that
\begin{equation}
    \label{eq:thm3_0}
    I_\xi(z)=\sum_{n\in\I} \sigma_n e^{-i\xi nL}\int_{\Lambda_t\cup\Lambda_{-t}}  \Phi_\xi(x,z)G_\xi(x,\x_0)ds(x).
\end{equation}
Note that $\{e^{i\xi_{(j)}t}\}_{j\in\Z}$ is an orthogonal basis in $L^2(-L/2,L/2)$ which means
$$\int_{-\frac{L}{2}}^{\frac{L}{2}} e^{i \xi_{(j_1)} t} \overline{e^{i \xi_{(j_2)}t}}dt = \begin{cases}
    L,\hspace{0.5 cm}  \text{if }j_1 = j_2, \\ \notag
     0,\hspace{0.6 cm}  \text{if }j_1 \neq j_2.
\end{cases}$$
Thus,
\begin{align}
    \notag
    &\int_{\Lambda_t} \Phi_\xi(x,z)G_\xi(x,\x_0)ds(x) \\ \notag
    &\hspace{3 cm}= \dfrac{1}{4L^2}\int_{-\frac{L}{2}}^{\frac{L}{2}} \sum_{j\in \Z \atop \beta_j>0} e^{-i\xi_{(j)}(x_1-z_1)-i\beta_j(t-z_2)} \sum_{l\in\Z}\dfrac{1}{\beta_l}e^{i\xi_{(l)}\left(x_1-\x_0^{(1)}\right)+i\beta_l\left(t-\x_0^{(2)}\right)} dx_1 \\ \label{eq:thm3_1}
    & \hspace{3 cm}= \sum_{j\in \Z \atop \beta_j>0}\dfrac{1}{4L\beta_j}  e^{i\xi_{(j)}\left(z_1-\x_0^{(1)}\right)+i\beta_j\left(z_2-\x_0^{(2)}\right)}. 
\end{align}
Similarly, on $\Lambda_{-t}$, one can also derive that
\begin{equation}\label{eq:thm3_2}
    \int_{\Lambda_{-t}} \Phi_\xi(x,z)G_\xi(x,\x_0)ds(x) = \sum_{j\in\Z \atop \beta_j>0} \dfrac{1}{4L\beta_j}e^{i\xi_{(j)}\left(z_1 -\x_0^{(1)}\right)-i\beta_j\left(z_2-\x_0^{(2)}\right)}.
\end{equation}
Substituting \eqref{eq:thm3_1} and \eqref{eq:thm3_2} into \eqref{eq:thm3_0} yields \eqref{thm:Iz}.

\smallskip 

{

\noindent To show the global maximum of $|I_\xi(z)|$, we first see from \eqref{thm:Iz} that
\begin{equation}
    |I_\xi(z)| = \left|\sum_{n\in \I} \sigma_n e^{-i\xi nL}\right| \left|\K_\xi(z, \x_0) \right|. 
\end{equation} Thus, it is equivalent to show that $ |\K_\xi(z, \x_0) |$ attains its global maximum uniquely at $z = \x_0$. Indeed, for each $ z\in \Omega_0$, 
\begin{align}
    \notag
    |\K_\xi(z)| \leq  \sum_{j\in\Z \atop \beta_j >0} \dfrac{1}{2L\beta_j}\left| e^{i\xi_{(j)}\left(z_1-\x_{0}^{(1)}\right)}\cos\left(\beta_j\left(z_2 -\x_{0}^{(2)}\right)\right)\right| 
     \le  \sum_{j\in\Z \atop \beta_j >0} \dfrac{1}{2L\beta_j}.
\end{align}
The maximum of $\left| I_\xi(z)\right|$ is obtained at $z=\x_0$.  Suppose that there exists a point $\x_* \in \Omega_0$ such that $|I_\xi(z)|$ also attains a maximum at $z=\x_*$. Then $\cos\left(\beta_j\left(\x_{*}^{(2)} - \x_{0}^{(2)}\right)\right)=\pm 1$ for all $j\in\Z$ with $\beta_j>0$ and
\begin{equation}\notag
    e^{i\xi_{(j)}\left(\x_{*}^{(1)} - \x_{0}^{(1)}\right)} = \pm 1 \, \text{ for all } j\in \Z, \, \beta_j>0 \quad \text{ or } \quad e^{i\xi_{(j)}\left(\x_{*}^{(1)} - \x_{0}^{(1)}\right)} = \pm i \, \text{ for all } j\in \Z, \, \beta_j>0.
\end{equation}
We first consider the $x_2$-direction.
For all $j\in\Z$ with $\beta_j>0$, we have
$$\beta_j\left(\x_{0}^{(2)} -\x_{*}^{(2)}\right)=n_2 \pi, \qquad n_2\in \Z.$$
 Taking $j=0$ and since $h<\tfrac{\pi}{2\sqrt{k^2-\xi^2}}$, we obtain that 
$$\left|\x_{0}^{(2)}-\x_*^{(2)}\right|=\dfrac{|n_2| \pi}{\sqrt{k^2-\xi^2}}>|n_2| 2h.$$
Thus, if $n_2\neq0$, then $\left|\x_{0}^{(2)}-\x_*^{(2)}\right|\ge 2h,$
which contradicts  the assumption that $\x_0^{(2)}, \x_*^{(2)} \in (-h, h)$. Therefore, $n_2=0$ and so $x_*^{(2)}=\x_0^{(2)}$. This implies that  $\cos\left(\beta_j\left(\x_{*}^{(2)} - \x_{0}^{(2)}\right)\right)=1$ for all $j\in\Z$ with $\beta_j>0$ and so
\begin{equation}\notag
    e^{i\xi_{(j)}\left(\x_{*}^{(1)} - \x_{0}^{(1)}\right)} = 1 \, \text{ for all } j\in \Z, \, \beta_j>0 \quad \text{ or } \quad e^{i\xi_{(j)}\left(\x_{*}^{(1)} - \x_{0}^{(1)}\right)} = -1 \, \text{ for all } j\in \Z, \, \beta_j>0,
\end{equation}
or 
\begin{equation}\notag
    e^{i\xi_{(j)}\left(\x_{*}^{(1)} - \x_{0}^{(1)}\right)} = i \, \text{ for all } j\in \Z, \, \beta_j>0 \quad \text{ or } \quad e^{i\xi_{(j)}\left(\x_{*}^{(1)} - \x_{0}^{(1)}\right)} = -i \, \text{ for all } j\in \Z, \, \beta_j>0.
\end{equation}
 For each case stated above, we take $j = 0$ and $j = 1$. Then
\[
    \xi \left(\x_{0}^{(1)} -\x_{*}^{(1)}\right) = \theta + 2n_1 \pi, \quad\text{and}  \quad \left(\xi + \frac{2\pi}{L}\right)\left(\x_{0}^{(1)} -\x_{*}^{(1)}\right) = \theta + 2n_1
'\pi,
\] where $\theta$ is the principal argument of $\pm 1$ or $\pm i$ and $n_1,\, n_1' \in \Z$. 
Then
\[
\frac{2\pi}{L} \left(\x_{0}^{(1)} -\x_{*}^{(1)}\right) = (n_1' - n_1) 2\pi,
\]which implies 
\[
    \left|\x_{0}^{(1)} -\x_{*}^{(1)}\right| =\frac{|(n_1' - n_1) 2\pi|L }{2\pi}=|n_1'-n_1|L.
\]
If $n_1'-n_1\neq 0$, then $\left|\x_{0}^{(1)}-\x_*^{(1)}\right|\ge L,$
which contradicts  the assumption that $\x_0^{(1)}, \x_*^{(1)} \in (-L/2, L/2)$. Therefore, $n_1'-n_1=0$ and so $x_*^{(1)}=\x_0^{(1)}$. Combining both components, we conclude that $|I_\xi(z)|$ attains its unique maximum at $z=\x_0$.

}
\end{proof}

{

 Theorem \ref{thm:imaging} provides a complete characterization of the source position $\x_0$ in the period $\Omega_0$ using $|I_{\xi}(z,\x_0)|$: the maximum of $|I_{\xi}(z,\x_0)|$ is attained uniquely at $z=\x_0$. Furthermore, this maximum forms a sharp peak, as $|I_{\xi}(z,\x_0)|$ remains small for $z \neq \x_0$. This property is essential for our numerical experiments, since it allows us to clearly distinguish the true source location even with oscillatory background.

This peak property can be explained by that of $|\K_\xi(z,\x_0)|$.
Figure \ref{figure1} illustrates that $|\K_\xi(z,\x_0)|$ has a significant peak at $z = \x_0$ and remains significantly smaller  as $z$ is away from $\x_0$. 
The decaying behavior can be explained by considering the real and imaginary parts of $\K_\xi(z,\x_0)$.
For the real part,
$$\text{Re}[\K_\xi(z,\x_0) ] = \sum_{j\in\Z \atop \beta_j >0} \dfrac{1}{2L\beta_j}\cos\left(\xi_{(j)}\left(z_1 -\x_{0}^{(1)}\right)\right)\cos\left(\beta_j\left(z_2 -\x_{0}^{(2)}\right)\right).$$
When $z \neq \x_0$, the cosine terms in the summation oscillate with different frequencies, causing cancellation in the summation. At the same time, for the imaginary part,
$$\text{Im}[\K_\xi(z,\x_0) ] = \sum_{j\in\Z \atop \beta_j >0} \dfrac{1}{2L\beta_j}\sin\left(\xi_{(j)}\left(z_1 -\x_{0}^{(1)}\right)\right)\cos\left(\beta_j\left(z_2 -\x_{0}^{(2)}\right)\right),$$
when $z\neq \x_0$, the sine and cosine terms oscillate at varying frequencies, again leading to mutual cancellation. Therefore, the magnitude of $\K_\xi(z,\x_0)$ remains significantly smaller when $z$ moves away from $\x_0$.
}

\begin{figure}[H]
    \centering
    \begin{subfigure}{0.32\textwidth}
        \centering
        \includegraphics[width=\textwidth]{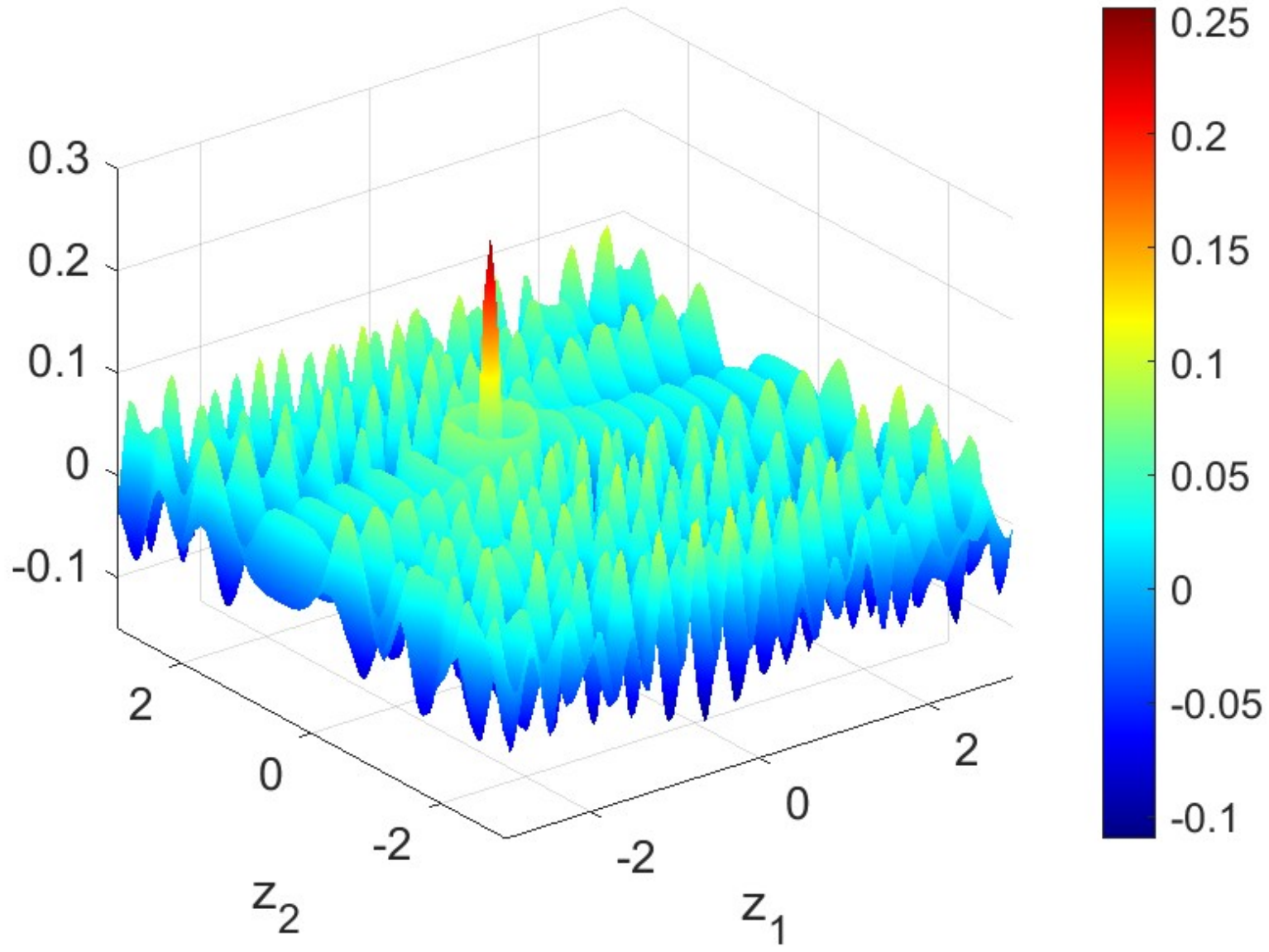}
        \caption{Re$\,[\K_\xi(z,\x_0)]$}
    \end{subfigure}
    \hfill
    \begin{subfigure}{0.32\textwidth}
        \centering
        \includegraphics[width=\textwidth]{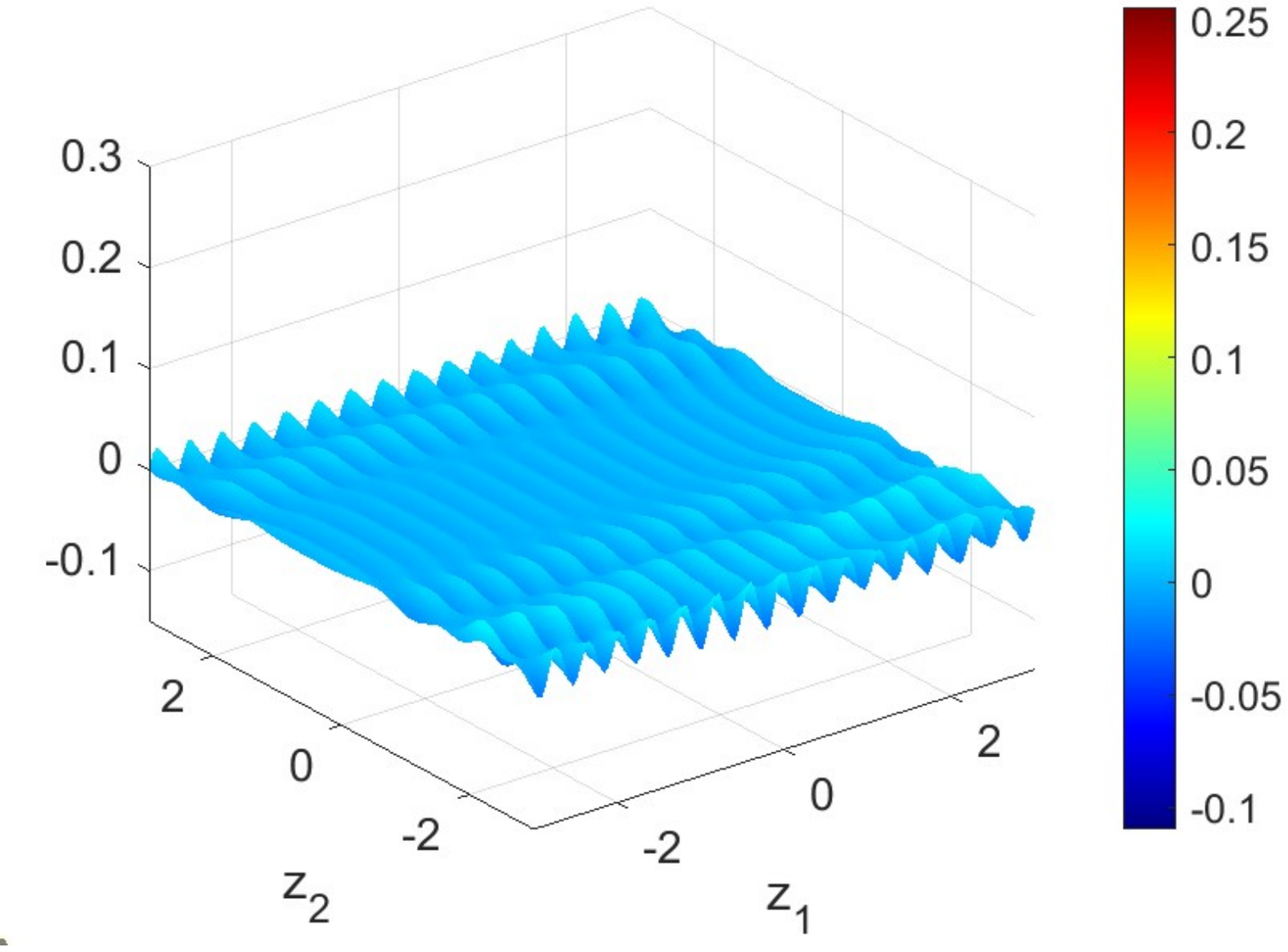}
        \caption{Im$\,[\K_\xi(z,\x_0)]$}
    \end{subfigure}
    \hfill
    \begin{subfigure}{0.32\textwidth}
        \centering
        \includegraphics[width=\textwidth]{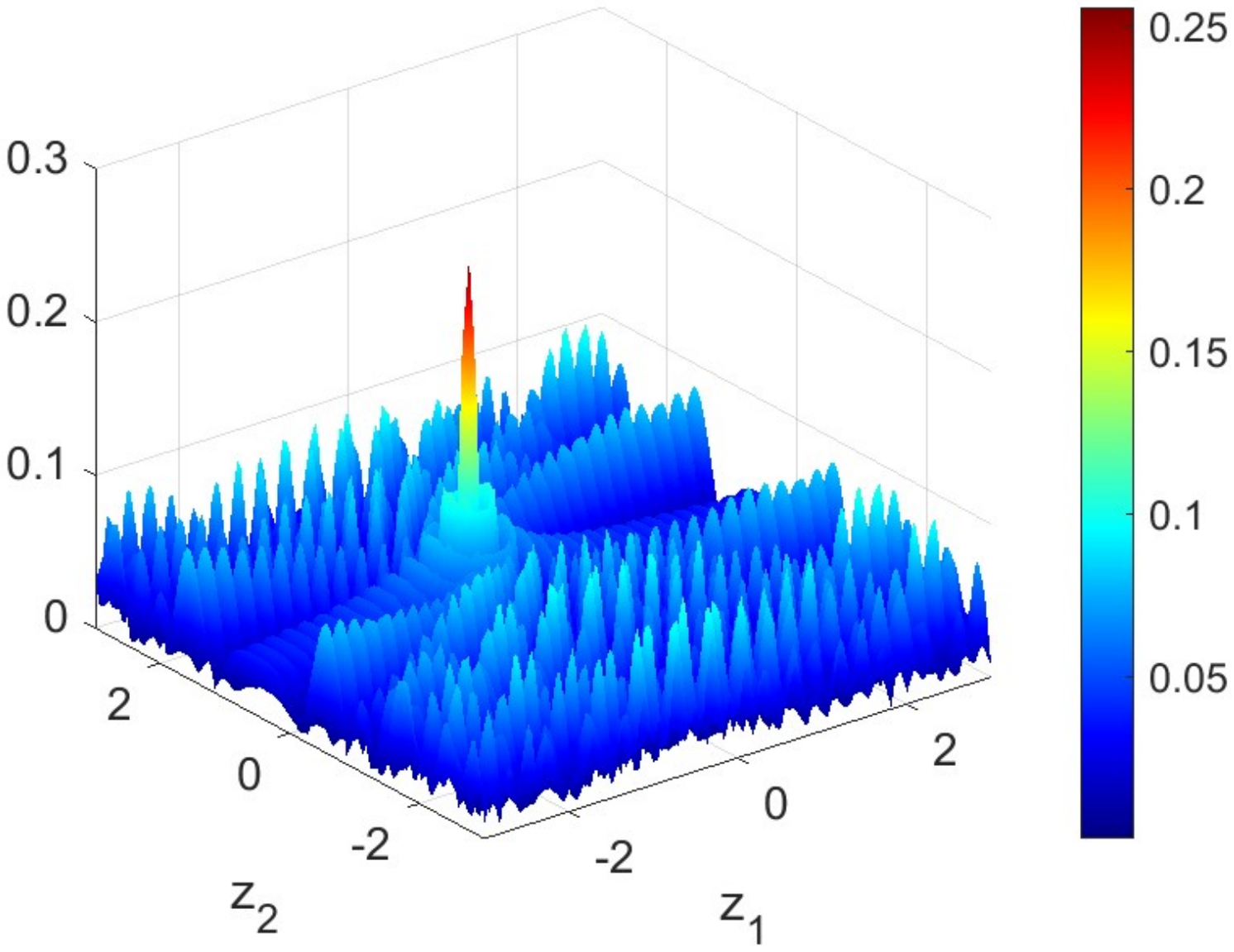}
        \caption{$|\K_\xi(z,\x_0)|$}
    \end{subfigure}

    \caption{$\text{Re}[\K_\xi(z,\x_0)],\ \text{Im}[\K_\xi(z,\x_0)],\ \text{and}\ |\K_\xi(z,\x_0)|$  
     at $\x_0 = (-0.5,\,0.5)$ with $\xi = 0.5$, $L=6$, and $k=16.5$.}
    \label{figure1}
\end{figure}

{
In practice, we use $|I_\xi|^p$, where the parameter $p>0$, to identify $\x_0$ and sharpen the peak.
}
Once $\x_0$ is determined, we denote 
\begin{equation}
\label{Theta}
\Theta_\xi = \sum_{n\in \I} \sigma_n e^{-i\xi nL},
\end{equation}
which can then be computed as
\begin{equation*}
    \Theta_\xi = \dfrac{I_\xi(\x_0)}{\K_\xi(\x_0,\x_0)}.
\end{equation*}
This $\Theta_\xi$ enables the detection of both locations and intensities of the defective sources, as demonstrated in the computational algorithm.

\subsection{Determination of $N$, $m$, and $\gamma_m$}
This section introduces a systematic approach, utilizing the imaging function $I_\xi(z)$ and algebraic techniques, to determine the number of defective units $N$, the indices $m \in \I$, and their associated intensities $\gamma_m$.

First, we develop the theorem to help determine the number of defective units. We employ a finite set of $\xi_j$ as in Theorem 
\ref{thm:uniquenessFB},
$$\xi_j = \xi_0 + j\Delta\xi,\qquad j=1,2,\ldots$$
with $0<
\Delta\xi\neq \tfrac{2\pi q}{L},\, q\in \mathbb{Q}$ such that $\txt \xi_j\in \left( -\frac{\pi}{L},\frac{\pi}{L}\right){ \setminus  \{0\}}$ for all $j$.
For $n=1,2,\ldots,N$, we denote further
$$a_n = (\gamma_{m_n}-1) e^{-i\xi_0 m_n L}, \qquad \text{and }\qquad r_n = e^{-i\Delta\xi m_nL}.$$
Recall from~\eqref{Theta} that $
    \Theta_\xi = \sum_{n\in \I} \sigma_n e^{-i\xi nL}$. Then
\begin{equation}\notag
    \Theta_{\xi_j}=a_1 r_1^j + a_2r_2^j+\ldots a_N r_N^j,
\end{equation}
where $N=|\I|$ and $m_n\in \I$.

\begin{theorem}\label{thm:nd} For $M\in \N^*$, define an  $M\times M$ matrix 
$$D_M:=\begin{bmatrix}
\Theta_{\xi_{M-1}} & \ldots & \Theta_{\xi_{0}}\\
\Theta_{\xi_{M}} & \ldots & \Theta_{\xi_{1}} \\ 
 \vdots & \ddots & \vdots \\
 \Theta_{\xi_{2M-2}} & \ldots & \Theta_{\xi_{M-1}}
\end{bmatrix}.$$
Then $ D_M$ are invertible  for all $M\le N$ and $D_M$ are singular otherwise if and only if there are $N$ defective sources.
\end{theorem}
\begin{proof}
    Assume that there are $N$ defective sources.  Let $M\in \mathbb{N}^*$ such that $M\le N$. As in the proof of Theorem~\ref{thm:uniquenessFB},  the Vandermonde matrix 
    $$ V =\begin{bmatrix}
    1 & r_1 & r_1^2 & \dots  & r_{1}^{M-1} \\
        1 & r_2 & r_2^2 & \dots  & r_2^{M-1} \\
    \vdots & \vdots & \vdots & \ddots & \vdots \\
       1 & r_M & r_M^2 & \dots  & r_{M}^{M-1} \\
\end{bmatrix}$$
is invertible. Let $v_j=\left(1,r_j,\ldots,r_j^{M-1}\right)^{\top}$ for $j=1,\ldots,M$, then $\{v_1,v_2,\ldots,v_M\}$ is linearly independent. Denote $$u_1=\begin{bmatrix}
    \Theta_{\xi_0}  \\
    \Theta_{\xi_1}  \\
    \vdots \\
    \Theta_{\xi_{M -1}} 
\end{bmatrix}=\begin{bmatrix}
    a_1+\ldots+a_M  \\
    a_1r_1+\ldots+a_M r_M  \\
    \vdots \\
    a_1r_1^{M-1}+\ldots+a_M r_M^{M-1}
\end{bmatrix}=a_1v_1+\ldots +a_M v_M,$$
$$\vdots$$
$$u_M=\begin{bmatrix}
    \Theta_{\xi_{M -1}}  \\
    \Theta_{\xi_{M}}  \\
    \vdots \\
    \Theta_{\xi_{2M -2}} 
\end{bmatrix}=\begin{bmatrix}
    a_1r_1^{M-1}+\ldots+a_M r_M^{M-1} \\
    a_1r_1^M+\ldots+a_M r_M^M  \\
    \vdots \\
    a_1r_1^{2M-2}+\ldots+a_M r_M^{2M-2}
\end{bmatrix}=a_1 r_1^{M-1}v_1+\ldots+ a_M r_M^{M-1} v_M.$$
Suppose that $\alpha_1 u_1+\ldots+\alpha_M u_M=0,$
then
$$\left(\alpha_1 a_1+\ldots +\alpha_M a_1 r_1^{M-1}\right)v_1+\ldots+\left(\alpha_1 a_M+\ldots+\alpha_M a_M r_M^{M-1}\right)v_M=0.$$
Since $\{v_1,v_2,\ldots,v_M\}$ is linearly independent and $a_n\neq 0$ for all $n=1,\ldots,M$, 
$$\begin{cases}
    \alpha_1 a_1+\ldots 
    +\alpha_M a_1 r_1^{M-1} = 0 \\
    \hspace{2 cm}\ldots  \\
    \alpha_1 a_M+\ldots+\alpha_M a_M r_M^{M-1} =0
\end{cases} \Rightarrow \begin{cases}
    \alpha_1 +\ldots +\alpha_M r_1^{M-1} = 0 \\
   \hspace{1.5 cm}\ldots  \\
    \alpha_1 +\ldots+\alpha_M  r_M^{M-1} =0
\end{cases} \Rightarrow V\begin{bmatrix}
    \alpha_1  \\
    \alpha_2  \\
    \vdots \\
    \alpha_{M} 
\end{bmatrix} =0.$$
Since $V$ is invertible, $\alpha_1=\alpha_2=\ldots=\alpha_M=0$. Therefore, $\{ u_1,u_2,\ldots,u_M\}$ is linearly independent, implying that $ D_M$ is invertible.
On the other hand, $$
    D_{N+1} = \begin{bmatrix}
\Theta_{\xi_N} & \ldots & \Theta_{\xi_0}\\
\Theta_{\xi_{N+1}} & \ldots & \Theta_{\xi_1} \\ 
 \vdots & \ddots & \vdots \\
 \Theta_{\xi_{2N}} & \ldots & \Theta_{\xi_N} 
\end{bmatrix}= \begin{bmatrix}
a_1& \ldots & a_{N+1}\\
a_1 r_1 & \ldots & a_{N+1}r_{N+1} \\ 
 \vdots & \ddots & \vdots \\
 a_1 r_1^N & \ldots & a_{N+1} r_{N+1}^N 
\end{bmatrix}\begin{bmatrix}
r_1^N& \ldots & 1\\
r_2^N & \ldots & 1 \\ 
 \vdots & \ddots & \vdots \\
r_{N+1}^N & \ldots & 1 
\end{bmatrix}. $$
Since 
$\gamma_{m_{N+1}}=1$, we have $a_{N+1}=0$. Thus
$$\begin{bmatrix}
a_1& \ldots & a_{N+1}\\
a_1 r_1 & \ldots & a_{N+1}r_{N+1} \\ 
 \vdots & \ddots & \vdots \\
 a_1 r_1^N & \ldots & a_{N+1} r_{N+1}^N 
\end{bmatrix}=\begin{bmatrix}
a_1& \ldots & 0\\
a_1 r_1 & \ldots & 0 \\ 
 \vdots & \ddots & \vdots \\
 a_1 r_1^N & \ldots & 0 
\end{bmatrix},$$
which implies that $ D_{N+1}$ is singular. Similarly, we can also show that $ D_{M}$ are singular for all $M>N+1$.

Now, assume that $D_M$ are invertible for $M\le N$ and $D_{M}$ are singular for all $M>N$ for some $N\in \N^*$ and there are $K\in\N^*$ defective sources. By the first part, $D_K$ is invertible and $D_{M}$ are singular for all $M>K$.
\begin{itemize}
    \item If $K<N$, then $D_N$ is singular. This contradicts our hypothesis.
    \item If $K>N$ defective sources, then the invertibility of $D_K$ contradicts the assumption that $D_{M}$ are singular for all $M>N$. 
\end{itemize}
Therefore, $K=N$.
\end{proof}

We now proceed to present the method for determining the indices and intensities of the defective sources.

\begin{theorem}\label{thm:mg}
    Let 
$N^*\in\N$ such that 
$\I\subset [-N^*,N^*]$. Then the linear system
\begin{equation}\label{systm:no}
     \begin{bmatrix}
e^{-i\xi_0 (-N^*) L} & \ldots & e^{-i\xi_0 N^* L} \\
 \vdots & \ddots & \vdots\\
e^{-i\xi_{2N^*} (-N^*) L} & \ldots & e^{-i\xi_{2N^*} N^* L}
\end{bmatrix}\begin{bmatrix}
\varkappa_1 \\
 \vdots\\
\varkappa_{2N^*+1}
\end{bmatrix} = \begin{bmatrix}
\Theta_{\xi_0} \\
 \vdots\\
\Theta_{\xi_{2N^*}}
\end{bmatrix}
\end{equation}
is uniquely solvable. 
Furthermore, the indices $m$ for the location of defective sources and their intensities $\gamma_m$ are determined by 
$$m=-N^*-1+j, \hspace{1 cm} \gamma_m=\varkappa_j+1,$$
for $j \in \{ n\in\N: 1 \leq n \leq 2N^* + 1, \varkappa_n \neq 0\}$.

\end{theorem}
\begin{proof}
    Similar to matrix $\mathcal{M}$ in \eqref{matrixM}, 
    $$\begin{bmatrix}
e^{-i\xi_0 (-N^*) L} & \ldots & e^{-i\xi_0 N^* L} \\
 \vdots & \ddots & \vdots\\
e^{-i\xi_{2N^*} (-N^*) L} & \ldots & e^{-i\xi_{2N^*} N^* L}
\end{bmatrix}$$ is invertible. Thus \eqref{systm:no} is uniquely solvable. Recall that 
$$\Theta_{\xi_j} = \sum_{m\in \I} (\gamma_m-1)e^{-i\xi_j mL}=\sum_{m\in[-N^*,N^*]\setminus \I} (\gamma_m-1)e^{-i\xi_j mL} + \sum_{m\in \I} (\gamma_m-1)e^{-i\xi_j mL}, $$
for $j=0,1,\ldots,2N^*$ with $\gamma_m=1$ for all $m\in [-N^*,N^*]\setminus \I$. 

This implies that $(\gamma_{-N^*}-1,\ldots,\gamma_{N^*}-1)$ is the solution of \eqref{systm:no}.
It follows that the restricted array $[-N^*, N^*]$ contains some normal sources, which correspond to the zero components of the solution $(\varkappa_1, \ldots, \varkappa_{2N^*+1})$ to the linear system~\eqref{systm:no}. 
Consequently, the defective sources can be identified through the nonzero components of this solution. 
In particular, if $\varkappa_j \neq 0$, the index and intensity of the corresponding defective source are given by
$$m=-N^*-1+j, \hspace{2 cm} \gamma_m=\varkappa_j+1,$$
 for some $j = 1, \ldots, 2N^* + 1$.
 \end{proof}

 The following algorithm is derived directly from Theorems \ref{thm:nd} and \ref{thm:mg}.
 
\begin{algorithm*}[H]\SetAlgoNlRelativeSize{0} 
\renewcommand{\thealgocf}{} 
\caption{Determination of defective sources}
Select a discrete subset $\mathcal{A} \subset \left(-\tfrac{\pi}{L}, \tfrac{\pi}{L}\right){\setminus \{0\}}$\;
\If{$\Theta_{\xi}=0$ for all $\xi\in\mathcal{A}$}{
        Conclude: no defective source\;
    }
\Else{
    Choose $\xi_0$ and  $\Delta\xi$\;
    Compute $\det D_1,\,\det D_2,\,  \det D_3, \ldots$ to determine $N$ - the number of defective sources\;
    Select $N^*\in \N$ such that $N^*>(N-1)/2$ and solve \eqref{systm:no} for $(\varkappa_{-N^*},\ldots,\varkappa_{N^*})$\;
 Set $\mathrm{X} = \{(\varkappa_j,j) : \varkappa_j \neq 0, \, j=-N^*,\ldots,N^*\}$\;
 \If{$|\mathrm{X}|\neq N$}{increase $N^*$ and repeat the procedure until $|\mathrm{X}|= N$\;}
 \Else{
Output all indices and intensities of defective sources
$$m=-N^*-1+j,\qquad \gamma_m = \varkappa_j+1,$$
for $(\varkappa_j,j)\in \mathrm{X}$.
}}
\end{algorithm*}
\begin{remark}
    We observe from the numerical implementation that if the chosen $N^*$ is not large enough to fully cover $\I$,  all components of the solution $(\varkappa_{-N^*},\ldots,\varkappa_{N^*})$ will be nonzero, leading to $2N^*+1$ defective sources. This is greater than the actual number of defects that was determined in the previous step since $N^*>(N-1)/2$. Such a result indicates that we should increase $N^*$.
\end{remark}

\section{Numerical examples}\label{section4}
We present several numerical examples to evaluate the performance of the proposed method. In all cases, 
$
\Omega_0 = (-1,1) \times (-2,2),$
with point sources arranged periodically with period $L = 2$. The data is generated by $$u_\xi=\mathcal{F}\left( \sum_{m \in \I} \sigma_m\frac{i}{4} H^{(1)}_0(k|x - \x_m|)\right)(\cdot;\xi),$$ 
for multiple values of $\xi\in \left(-\tfrac{\pi}{L},\tfrac{\pi}{L} \right)\setminus\{0\}$ with a note that $\mathcal{F}G_{\text{per}}(\cdot;\xi)=0$ for $\xi\in \left(-\tfrac{\pi}{L},\tfrac{\pi}{L} \right)\setminus\{0\}$.
Here, $u_\xi$ is computed at wave number $k=10.5$ on $\Lambda_{3} \cup \Lambda_{-3}$ with added artificial random noise.  The noise vector $\mathcal{N}$ is incorporated to $u_\xi$, the solution to \eqref{eq:direct:1s0}, as follows
$$u_\xi+\delta\dfrac{\mathcal{N}}{\| \mathcal{N}\|_F}\| u_\xi\|_F,$$
where $\delta$ is the level of noise and $\| \cdot\|_F$ is the Frobenius matrix norm. Note that $\mathcal{N}$ consists of numbers $a+bi$ where $a,b\in (-1,1)$ are randomly generated with a uniform distribution. 
More specifically, we use
$$\xi_j = -\dfrac{\pi}{L} + 0.01 + 0.295j,\qquad j=0,1,\ldots,8$$
for computing $\det D_M$ and
$$\xi_j = -\dfrac{\pi}{L} + 0.03 + 0.13j,\qquad j=0,1,\ldots,22$$
for reconstructing $m$ and $\gamma_m$.

We can clearly see a significant peak of $|I_\xi(z)|^3$ (i.e., $p=3$) in Figure \ref{pic2}. Then we determine $\x_0$ as the location of this peak. Tables \ref{tb21} and \ref{tb22} demonstrate that the computed location of $\x_0$ is highly accurate.

\begin{figure}[H]
    \centering
    \begin{subfigure}{0.33\textwidth}
        \centering
        \includegraphics[width=\textwidth]{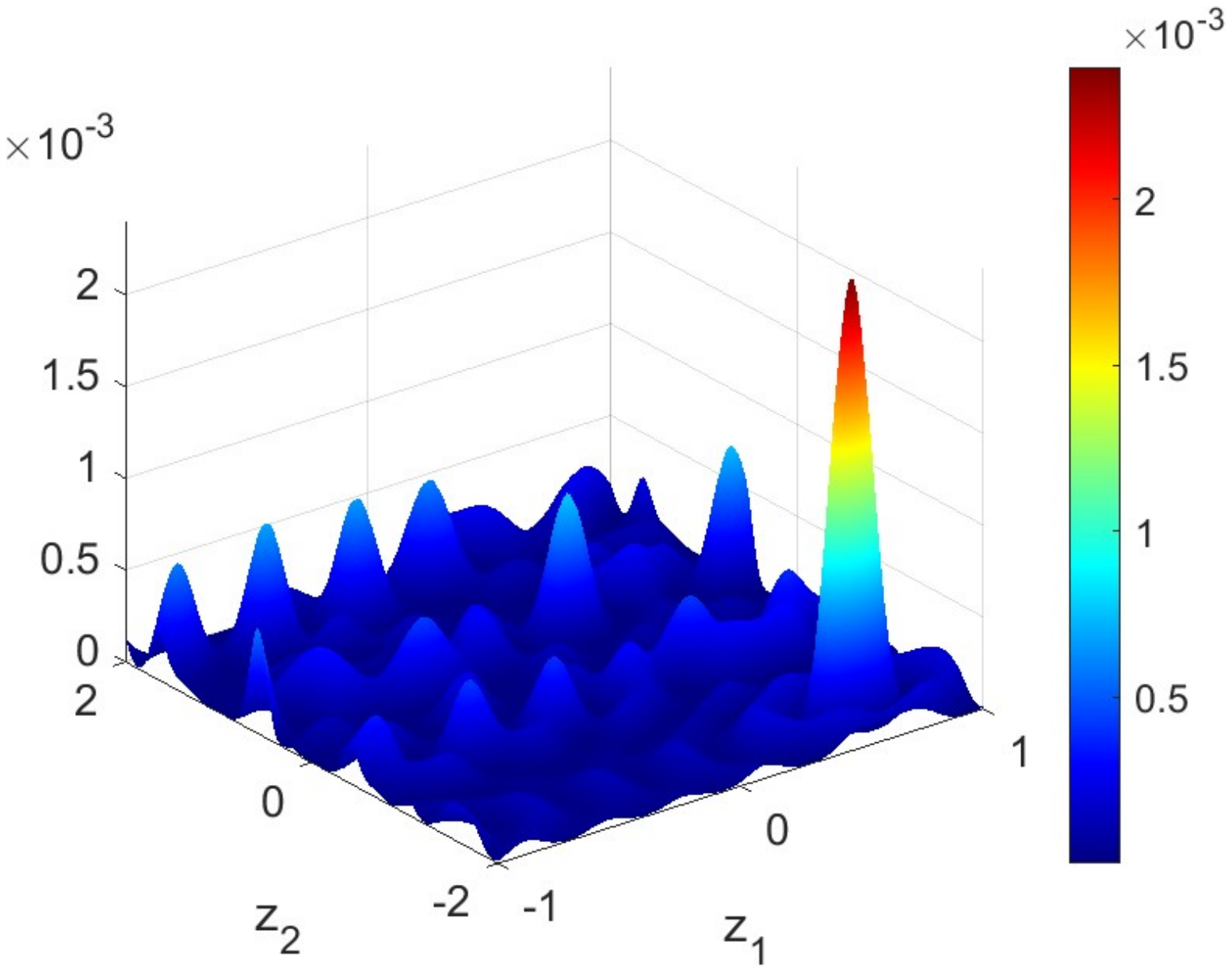}
        \caption{$\x_0 = (0.61,-1.59)$}
        \label{fig:pic2-left}
    \end{subfigure}
    \hspace{0cm}
    \begin{subfigure}{0.33\textwidth}
        \centering
        \includegraphics[width=\textwidth]{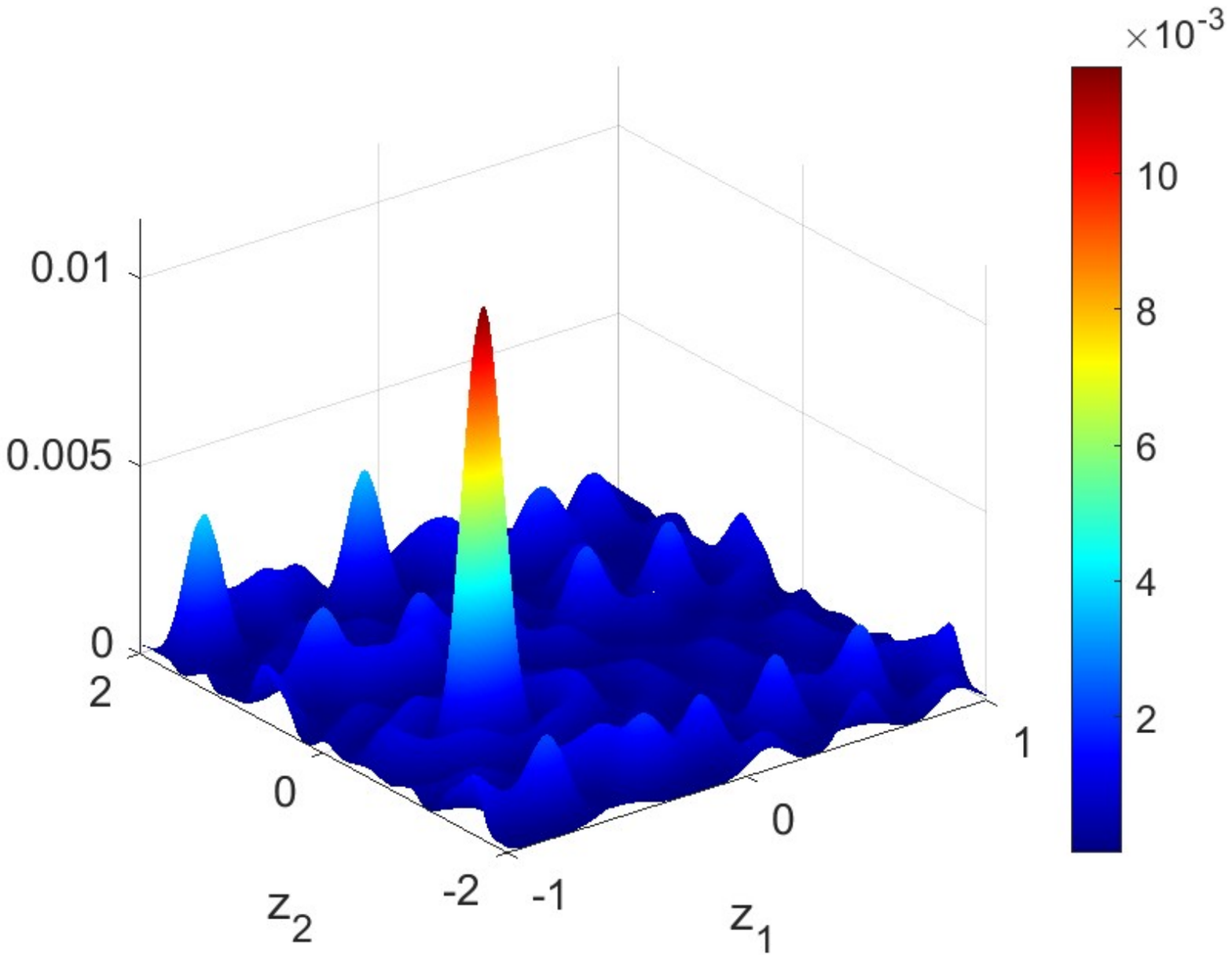}
        \caption{$\x_0 = (-0.52,-0.48)$}
        \label{fig:pic2-right}
    \end{subfigure}
    \caption{$|I_\xi(z)|^3$ with $\xi = -0.67$.}
    \label{pic2}
\end{figure}

Tables \ref{table:11} and \ref{table:12} summarize the computational results for determining the number of defective units by $\det D_1,\,\det D_2,\, \det D_3,\ldots$. Following Theorem \ref{thm:nd}, if $\det D_M \neq 0$ for $M\le N$ and $\det D_{M}= 0$ for $M >N$, there are $N$ defective sources. This pattern can be observed that in all test cases, $\det D_M$ remain nonzero when $M$ does not exceed the true number of defects and $\det D_M$ become nearly zero for larger $M$, confirming the validity and accuracy of the proposed criterion.

\begin{table}[H]
\centering
\footnotesize
\begin{tabular}{cc}
\hline
True number of defective sources                                               & $\det D_M$                                                                                                                                                                        \\ \hline
1                                                      & \begin{tabular}[l]{@{}l@{}l@{}}$\det D_1 =0.5125 - 0.0298i$\\$\det D_2 = -6.5\cdot 10^{-4} - 4.2\cdot 10^{-4}i$\\ $\det D_3 = -2.0\cdot 10^{-4}- 1.8\cdot 10^{-4}i$
\end{tabular}  

\\ \hline
2        & \begin{tabular}[l]{@{}l@{}l@{}}$\det D_1 = 0.8014 - 0.0922i
$\\$\det D_2 =  0.0793 + 0.2970i$\\ $\det D_3 = 0.0007 - 0.0124i$\\ 
$\det D_4 = 0.0017 + 0.0004i$
\end{tabular}

\\ \hline
3 & \begin{tabular}[l]{@{}l@{}l@{}}$\det D_1 =  0.9240 + 0.0455i$\\$\det D_2 = 0.4889 - 0.6922i$\\ $\det D_3 =0.0496 + 0.1050i$\\ $\det D_4 = -0.0036 + 0.0065i
$
\\
$\det D_5 = 9.9\cdot 10^{-5} - 2.0\cdot 10^{-4}i$
\end{tabular}                                              \\ \hline
\end{tabular}
\captionof{table}{Determination of the number of defective sources (noise level $\delta = 10\%$)}\label{table:11}
\end{table}

 \begin{table}[H]
\centering
\footnotesize
\begin{tabular}{cc}
\hline
True number of defective sources                                               & $\det D_M$                                                                                                                                                                        \\ \hline
1                                                      & \begin{tabular}[l]{@{}l@{}l@{}}$\det D_1 =0.5132+ 0.0326i$\\$\det D_2 = -0.0016 + 0.0067i$\\ $\det D_3 = 5.5\cdot 10^{-5}- 2.1\cdot
 10^{-4}i
$
\end{tabular}  

\\ \hline
2        & \begin{tabular}[l]{@{}l@{}l@{}}$\det D_1 = 0.8006 - 0.0895i
$\\$\det D_2 =  0.0528 + 0.3013i$\\ $\det D_3 = 0.0006 - 0.0212i$\\ 
$\det D_4 = 0.0001 - 0.0029i$
\end{tabular}

\\ \hline
3 & \begin{tabular}[l]{@{}l@{}l@{}}$\det D_1 =  0.9305 + 0.0399i$\\$\det D_2 = 0.4343 - 0.7450i$\\ $\det D_3 =
  0.1279 - 0.0863i$\\ $\det D_4 = -0.0026 + 0.0080i
$
\\
$\det D_5 = 0.0070 - 0.0041i$
\end{tabular}                                              \\ \hline
\end{tabular}
\captionof{table}{Determination of the number of defective sources (noise level $\delta = 20\%$)}\label{table:12}
\end{table}

The reconstruction results for the indices and intensities of the defective point sources are presented in Tables \ref{tb21} and \ref{tb22}. All indices are successfully recovered. The reconstructed intensities are also accurate, with relative errors below $5\%$ in all cases. It is worth noting that even under $20\%$ noise, the method maintains reasonable accuracy, with relative errors in the reconstructed intensities remaining within $13\%$.

\begin{table}[H]\centering
\footnotesize
\begin{tabular}{cccccc}
\hline
True $\x_0$      & Comp. $\x_0$  & True $m$                                            & Comp. $m$                                                            & True $\gamma$                                       & Comp. $\gamma$                                          \\ \hline
$(0.61, -1.59)$ & $(0.6063, -1.5906)$ & \begin{tabular}[c]{@{}c@{}}$3$\end{tabular}   & \begin{tabular}[c]{@{}c@{}}$3$\end{tabular}   & \begin{tabular}[c]{@{}c@{}}$0.5$\end{tabular} & \begin{tabular}[c]{@{}c@{}}$0.5236 + 0.0292i
$ \end{tabular}\\
 \hline
 $(-0.70, -1.15)$ & $(0.7008, -1.1496)$ & \begin{tabular}[c]{@{}c@{}}$-2$\\ $5$\end{tabular}   & \begin{tabular}[c]{@{}c@{}}$-2 $\\ $5$\end{tabular}   & \begin{tabular}[c]{@{}c@{}}$0.9$\\ $0.1$\end{tabular} & \begin{tabular}[c]{@{}c@{}}$0.8970 - 0.0129i$\\ $0.0972 - 0.0114i
$\end{tabular}\\
 \hline
 $(-0.52, -0.48)$ & $( 0.5118,  -0.4882)$ & \begin{tabular}[c]{@{}c@{}c@{}}$-1 $\\ $3$\\ $10$\end{tabular}   & \begin{tabular}[c]{@{}c@{}c@{}}$-1 $\\ $3$\\ $10$\end{tabular}   & \begin{tabular}[c]{@{}c@{}c@{}}$0.3$\\ $0.5$\\$0.7$\end{tabular} & \begin{tabular}[c]{@{}c@{}c@{}}$0.2924 - 0.0064i$\\ $0.4940 + 0.0022i$ \\ $0.6905 - 0.0068i$\end{tabular}\\
 \hline
\end{tabular}
\captionof{table}{Reconstruction results for the indices and intensities of defective sources (noise level $\delta = 10\%$)}\label{tb21}
\end{table}

\begin{table}[H]\centering
\footnotesize
\begin{tabular}{cccccc}
\hline
True $\x_0$      & Comp. $\x_0$  & True $m$                                            & Comp. $m$                                                            & True $\gamma$                                       & Comp. $\gamma$                                          \\ \hline
$(0.61, -1.59)$ & $(0.6063, -1.5906)$ & \begin{tabular}[c]{@{}c@{}}$3$\end{tabular}   & \begin{tabular}[c]{@{}c@{}}$3$\end{tabular}   & \begin{tabular}[c]{@{}c@{}}$0.5$\end{tabular} & \begin{tabular}[c]{@{}c@{}}$0.5195 + 0.0172i$ \end{tabular}\\
 \hline
 $(-0.70, -1.15)$ & $(0.7008, -1.1496)$ & \begin{tabular}[c]{@{}c@{}}$-2$\\ $5$\end{tabular}   & \begin{tabular}[c]{@{}c@{}}$-2 $\\ $5$\end{tabular}   & \begin{tabular}[c]{@{}c@{}}$0.9$\\ $0.1$\end{tabular} & \begin{tabular}[c]{@{}c@{}}$0.8964+ 0.0096i$\\ $0.0870 - 0.0099i

$\end{tabular}\\
 \hline
 $(-0.52, -0.48)$ & $( 0.5118,  -0.4882)$ & \begin{tabular}[c]{@{}c@{}c@{}}$-1 $\\ $3$\\ $10$\end{tabular}   & \begin{tabular}[c]{@{}c@{}c@{}}$-1 $\\ $3$\\ $10$\end{tabular}   & \begin{tabular}[c]{@{}c@{}c@{}}$0.3$\\ $0.5$\\$0.7$\end{tabular} & \begin{tabular}[c]{@{}c@{}c@{}}$0.2789 - 0.0259i$\\ $0.4775 - 0.0143i$ \\ $0.6848 + 0.0081i$\end{tabular}\\
 \hline
\end{tabular}
\captionof{table}{Reconstruction results for the indices and intensities of defective sources (noise level $\delta = 20\%$)}\label{tb22}
\end{table}

\section{Conclusion}\label{section5}
In this work, we investigated the problem of identifying defective sources for Helmholtz equation in unbounded periodic arrays. The main challenge is that defects break the periodic pattern, which makes classical inverse source methods ineffective. To address this, we used the Floquet–Bloch transform to reformulate the original problem into a quasi-periodic inverse source problem.
We proved that both the original and the quasi-periodic problems have unique solutions, meaning the defective sources can be determined from boundary measurements. Building on this foundation, we proposed a numerical reconstruction method that combines a stable imaging function with an algebraic identification scheme. This hybrid method reliably recovers the number, locations, and intensities of defective sources. The numerical experiments show that the approach is accurate, easy to implement, computationally efficient, and highly stable even under high noise levels. 
Future work may extend the framework to small-volume source models and to inverse problems governed by Maxwell's equations.

\vspace{0.3cm}

\noindent {\bf Acknowledgment.} The authors would like to thank the anonymous referees for their careful reading and insightful comments, which have  improved the quality of the presentation. The research of T.-P. Nguyen is partially supported by the NSF Grant DMS-2513647. The research of D.-L. Nguyen and N. H. Nguyen is partially supported by NSF Grant DMS-2243854.
\bibliographystyle{abbrv}
\bibliography{references} 

\end{document}